\documentclass[11pt]{article}
\usepackage{amsmath}
\usepackage{hyperref}
\usepackage{amsmath}
\usepackage{amsfonts}
\usepackage{graphicx}
\usepackage{amssymb}
\usepackage{leftidx}
\usepackage{color}

\allowdisplaybreaks
\newtheorem{theorem}{Theorem}[section]
\newtheorem{lemma}[theorem]{Lemma}

\newtheorem{example}[theorem]{Example}
\newtheorem{definition}[theorem]{Definition}

\newtheorem{remark}[theorem]{Remark}

\numberwithin{equation}{section}

\newenvironment{proof}[1][Proof]{\noindent\textbf{#1.} }{\ \rule{0.5em}{0.5em}}

\begin{document}

\title{ Dynamic Programming for Indefinite Stochastic McKean-Vlasov LQ
Control Problem under Input Constraints \\ \small{Continuously updated}}
\author{Xun Li$^{1}$\thanks{
X. Li acknowledges the financial support by the Hong Kong General Research
Fund under grants 15216720, 15221621 and 15226922. E-mail:
li.xun@polyu.edu.hk}, \ \ Liangquan Zhang$^{2}\thanks{%
L. Zhang acknowledges the financial support partly by the National Nature
Science Foundation of China (Grant No. 12571486 and 12171053) and the
Fundamental Research Funds for the Central Universities, and the Research
Funds of Renmin University of China (No. 23XNKJ05).}$ \\
{\small 1. Department of Applied Mathematics }\\
{\small \ The Hong Kong Polytechnic University, Hong Kong, China }\\
{\small 2. School of Mathematics }\\
{\small \ Renmin University of China, Beijing 100872, China }}
\maketitle

\begin{abstract}

In this paper, motivated by the prohibition of short-selling in the
mean-variance portfolio selection problem, we study a class of indefinite
stochastic McKean-Vlasov linear quadratic (LQ) control problems under
nonnegative control constraints. In contrast to the conventional setting,
both the classical dynamic programming principle (DPP) and the standard
Riccati equation approach (based on completion of squares) fail due to the
presence of control constraints. We overcome this obstacle by employing the
Kuhn-Tucker conditions (see Xu and Shreve \cite{XuShreve2}) and a convex
analysis approach (see Rockafellar \cite{Rock}) to derive the optimal
control via an infinite-dimensional Hamilton-Jacobi-Bellman (HJB) equation.
The value function can be obtained through four sets of novel ordinary
differential equations (ODEs). An application to continuous time
mean-variance portfolio selection problems is presented, in which both the
investment risk and the capital market line are simultaneously determined.
\end{abstract}




\section{Introduction}


Since the pioneering works on McKean--Vlasov equations were introduced by
McKean Jr. \cite{mck} and Kac \cite{kac1, kac2}, a substantial body of
literature has emerged focusing on uncontrolled stochastic differential
equations (SDEs) and establishing general propagation of chaos results.
Significant attention has been devoted in recent decades to their connection
with mean-field game (MFG) theory, which was independently and
simultaneously developed by Lasry and Lions \cite{ll07} and by Huang,
Caines, and Malham\'{e} \cite{hcm}. The McKean--Vlasov equation arises
naturally when analyzing the behavior of large populations of symmetric
agents who interact through the empirical distribution of their states,
particularly in the search for a Nash equilibrium (competitive equilibrium)
or a Pareto equilibrium (cooperative equilibrium) (see \cite{bfy13, cdl13}).


The classical DPP for optimal control problems of McKean--Vlasov
type-referred to as mean field in some literature---fails to hold due to the
presence of the law of the process in the coefficients and its nonlinear
dependence\footnote{%
Whenever the objective function in the type like $\mathbb{E}\left[ U\left(
x\left( T\right) \right) \right] $, the dynamic programming is \emph{%
applicable} due to the so-called \textquotedblleft smoothing
property\textquotedblright\
\begin{equation*}
\mathbb{E}\left[ \mathbb{E}\left[ U\left( x\left( T\right) \right)
\left\vert \mathcal{F}_{m}\right. \right] \left\vert \mathcal{F}_{n}\right. %
\right] =\mathbb{E}\left[ U\left( x\left( T\right) \right) \left\vert
\mathcal{F}_{n}\right. \right] ,
\end{equation*}%
where $\left\{ \mathcal{F}_{k}\right\} _{k=1,2\ldots }$ is the underlying
filtration and $n\leq m$.
\par
However, for $U\left( \mathbb{E}\left[ x\left( T\right) \right] \right) ,$
no analogous relation holds (e.g. mean variance $\left[ \mathbb{E}x\left(
T\right) \right] ^{2}$), such as
\par
\begin{equation*}
\mathbb{E}\left[ U\left( \mathbb{E}\left[ x\left( T\right) \left\vert
\mathcal{F}_{m}\right. \right] \right) \left\vert \mathcal{F}_{n}\right. %
\right] \neq U\left( \mathbb{E}\left[ x\left( T\right) \left\vert \mathcal{F}%
_{n}\right. \right] \right) .
\end{equation*}%
} with respect to it. For instance, the classical mean--variance portfolio
selection problem in finance (see \cite{YZ1999, zl2000}) exemplifies such a
case. Therefore, problems of this nature fall within a time-inconsistent
framework (see Bj\"{o}rk, Khapko and Murgoci \cite{bkm17}; Bj\"{o}rk, Khapko
and Murgoci \cite{bm14}; Hern\'{a}ndez \& Possama\"{\i} \cite{HP20} and
references therein). Although the problem itself is time-inconsistent, it is
possible to recover a form of the DPP by extending the state space-for
example, as shown by Lauri\`{e}re and Pironneau \cite{LP14}. Bensoussan,
Frehse and Yam \cite{bfy13, bfy15, bfy17} assumed that the marginal
distribution of the state process admits a density at all times and
transformed the original problem into a density control problem governed by
a family of deterministic controls. They then established a DPP and derived
the corresponding Bellman equation in the space of density functions. Pham
and Wei \cite{PW18} obtained a DPP for closed-loop controls. For open-loop
controls, related results can be found in Cosso and Pham \cite{CH2019} for
McKean--Vlasov differential games and in Bayraktar, Cosso and Pham \cite%
{BCP2018} for the so-called randomised DPP. When common noise is present,
Pham and Wei \cite{PW2017} established a DPP under which the control process
is adapted to the filtration generated by the common noise. Furthermore,
Bouchard, Djehiche and Kharroubi \cite{BDK2020} studied a stochastic
McKean--Vlasov target problem in which the controlled process satisfies
prescribed marginal constraints and established a general geometric DPP (see
also \cite{ST2002}).

Historically, the theory of indefinite stochastic LQ control has been
extensively developed and has found numerous significant applications. Chen,
Li and Zhou \cite{clz98} studied a class of indefinite LQ problems based on
the Riccati equation approach. Ait Rami, Moore and Zhou \cite{amz01}
demonstrated that the solvability of the generalized Riccati equation is
both necessary and sufficient for the well-posedness of the indefinite LQ
problem. Subsequent studies have extended these results to various settings;
see, for example, Kohlmann and Zhou \cite{kz00} and Qian and Zhou \cite{qz13}%
. In the positive definite case, mean field LQ problems have been actively
investigated over the past decade. Yong \cite{yong13} analyzed a mean field
LQ problem with deterministic coefficients over a finite time horizon and
derived the optimal feedback control via a system of Riccati equations.
Further developments building upon Yong (2013) can be found in \cite{sms14,
hly19, lsy16, sw21, y17}, as well as in the monograph by Sun and Yong \cite%
{sy20}. More recently, Li, Li and Yu \cite{lly20} investigated indefinite
stochastic mean-field LQ optimal control problems, in which the weighting
matrices for the state and control variables in the cost functional are
allowed to be indefinite.

%

Indeed, the indefinite McKean-Vlasov LQ problems arise from the
mean-variance portfolio selection problem. Markowitz initially proposed and
solved this problem in a single-period setting in his Nobel Prize-winning
work (Markowitz, 1952, 1959 \cite{Markowitz}), which laid the foundation for
modern finance. Subsequently, this model was extended to multi-period and
continuous-time portfolio selection frameworks. Whenever one attempts to
solve the mean-variance portfolio selection problem, two objectives must be
addressed: minimizing the variance of terminal wealth (i.e., the difference
between terminal wealth and its expected value), and maximizing the expected
terminal wealth. Since these two criteria are combined within a single cost
functional, this stochastic control problem differs fundamentally from the
classical LQ problem. The primary reason lies in the nonlinear term $(%
\mathbb{E}[X(T)])^2$. Li and Zhou (2000) embedded this problem into an
auxiliary stochastic LQ framework, which in fact belongs to the class of
indefinite LQ problems.


As mentioned before, due to the presence of the law of the state in a
nonlinear way, for instance, the term $(\mathbb{E}[X(T)])^{2}$, the problem
belongs to time-inconsistent in the sense that an optimal control viewed
from today is no more optimal when viewed from tomorrow, and this would
prevent a priori the use of the classical dynamic programming method. To
handel the time inconsistency, one in turn adopts typically either
pre-commitment strategies, i.e. controls that are optimal for the problem
viewed at the initial time, but may be not optimal at future date, or
game-equilibrium strategies, that is, control decisions regarded as a game
against all the future decisions the controller is going to make. Recently,
Pham and Wei \cite{PW2017, PW18} developed a series of approaches
emphasizing that time consistency can be actually restored for
pre-commitment strategies, provided that one considers as state variable the
conditional law of the state process instead of the state itself, therefore
making possible the use of the dynamic programming method. More preciously,
by general form of It\^{o} formula (cf. \cite{gpw23, cd17}) it is successful
to give a derivation of infinite dimensional H-J-B equation. Cosso, Gozzi,
Kharroubi, Pham and Rosestolato (cf. \cite{cgkpr23}) studied the optimal
control of path-dependent McKean-Vlasov equations valued in Hilbert spaces
motivated by non Markovian mean-field models driven by stochastic PDEs and
derived the similar results. %

In this paper, we focus on McKean-Vlasov LQ control under control
constraints. This framework exhibits several distinct characteristics in
comparison to the aforementioned works, as outlined below:

\begin{itemize}
\item The McKean-Vlasov LQ problem discussed above requires the control
variable to take values in the entire space, which is somewhat restrictive
from an application standpoint-for instance, in the continuous-time
mean-variance portfolio selection problem where short-selling of stocks is
prohibited. As a result, the elegant Riccati approach\footnote{%
The standard LQ theory typically assumes unconstrained controls, as the
optimal control derived from the Riccati equation may not satisfy prescribed
control constraints.} cannot be directly applied, except in certain special
frameworks (see Hu and Zhou \cite{hz05} for the homogeneous case).

\item The maximum principle (or necessary optimality condition) for optimal
control involves an adjoint equation-under a convex control domain-and
subsequently characterizes the optimal control through the Hamiltonian
system, which requires the control weight matrix $R$ to be non-singular. In
our setting, however, $R \equiv 0 $. Moreover, the presence of the nonlinear
term $(\mathbb{E}[X(T)])^2 $ renders the classical dynamic programming
approach ineffective. To address these challenges, we circumvent the
standard framework by employing the Kuhn-Tucker conditions (see Lemma 3.2 in
Xu and Shreve \cite{XuShreve2}) and a convex analysis approach (see
Rockafellar \cite{Rock}) to derive the optimal control via an
infinite-dimensional HJB equation\footnote{%
In fact, the Riccati equation can be regarded as the HJB equation with
separated time and spatial variables.}. As in \cite{PW2017} and \cite%
{cgkpr23}, we treat the probability distribution as a functional argument.
The advantage of this formulation lies in its direct handling of the
variance minimization of $X(T) $, thereby avoiding the need for introducing
and analyzing Lagrange multipliers.

\item In contrast to Li, Zhou and Lim \cite{LZL2002}, we derive four sets of
ODEs, the second of which (see (\ref{ode2})) is a standard Riccati equation
for which an explicit solution is generally not available. Based on these
results, we revisit the continuous time mean-variance portfolio selection
problem under the prohibition of short-selling stocks and simultaneously
determine both the investment risk and the capital market line.
\end{itemize}

The outline of this paper is as follows. After the preliminaries in Section %
\ref{sect2}, we study a stochastic McKean-Vlasov LQ control problem under
constraints and derive the optimal feedback control in Section \ref{sect3}.
As an application, in Section \ref{sect4}, we revisit the mean-variance
portfolio selection problem under short-selling prohibition and determine
both the investment risk and the capital market line. Section \ref{sect5}
concludes the paper. Additionally, several well-known results are presented
in Appendix \ref{sect6}.

\section{Problem Formulation and Preliminaries}

\label{sect2}

\subsection{Notation}

We make use of the following notation:
\begin{equation*}
\begin{tabular}{crl}
$M^{\prime }$ & : & the transpose of any matrix or vector $M$; \\
$\Vert M\Vert $ & : & $\sqrt{\sum_{i,j}m_{ij}^{2}}$ for any matrix or vector
$M=(m_{ij})$; \\
$\mathbb{R}^{n}$ & : & $n$ dimensional real Euclidean space; \\
$\mathbb{R}_{+}^{n}$ & : & the subset of $\mathbb{R}^{n}$ consisting of
elements with nonnegative components.%
\end{tabular}%
\end{equation*}%
Consider a finite time horizon $[0,T]$ for fixed $T>0$. We assume $(\Omega ,%
\mathcal{F},\{\mathcal{F}_{t}\}_{0\leq t\leq T},P)$ is a complete, filtered
probability space on which a standard $m$-dimensional Brownian motion $%
W(t)\equiv (W^{1}(t),\cdots ,W^{m}(t))^{\top }$ is defined. We assume that
there exists a sub-$\sigma $-algebra $\mathcal{G}$ of $\mathcal{F}$, with $%
\mathcal{G}$ \textquotedblleft rich enough\textquotedblright\ as will be
explained later. Moreover, we will adopt the following notations, unless
otherwise specified.

\begin{itemize}
\item Given a probability space $(\Omega ,\mathcal{F},P)$ with a filtration $%
\{\mathcal{F}_{t}|a\leq t\leq b\}(-\infty \leq a<b\leq +\infty )$, a Hilbert
space $\mathcal{H}$ with the norm $\Vert \cdot \Vert _{\mathcal{H}}$, define
the Banach space
\begin{equation*}
L_{\mathcal{F}}^{2}(0,T;\mathcal{H})\triangleq \left\{ \varphi (\cdot
)\left\vert
\begin{array}{l}
\varphi (\cdot )\text{ is an }\mathcal{F}_{t}\text{-adapted, }\mathcal{H}%
\text{-valued measurable } \\
\text{process on }[a,b]\text{ and }\mathbb{E}\left[ \int_{a}^{b}\Vert
\varphi (t,\omega )\Vert _{\mathcal{H}}^{2}\mathrm{d}t\right] <+\infty%
\end{array}%
\right. \right\}
\end{equation*}%
with the norm $\Vert \varphi (\cdot )\Vert _{\mathcal{F},2}=\Big [\mathbb{E}%
\int_{a}^{b}\Vert \varphi (t,\omega )\Vert _{\mathcal{H}}^{2}\mathrm{d}t%
\Big
]^{\frac{1}{2}}<+\infty .$ Additionally, let $\mathcal{H}^{p}\left[ 0,T%
\right] =\mathcal{L}_{\mathcal{F}}^{p}\left( \Omega ;C\left( \left[ 0,T%
\right] ;\mathbb{R}^{n}\right) \right) $ with $\left\Vert \varphi \left(
\cdot \right) \right\Vert _{\mathcal{H}^{p}}=\left[ \mathbb{E}\left(
\sup_{t\in \left[ 0,T\right] }\left\vert \varphi \left( t\right) \right\vert
^{p}\right) \right] ^{1\wedge \frac{1}{p}},$ $\forall \varphi \in \mathcal{H}%
^{p}\left[ 0,T\right] .$

\item A set $\mathcal{U}_{ad}^{p,+}\left[ 0,T\right] $ of admissible
controls is defined by
\begin{equation*}
\mathcal{U}_{ad}^{p,+}\left[ 0,T\right] \triangleq \left\{ u\left( \cdot
\right) \left\vert
\begin{array}{l}
u\left( t\right) \in \mathbb{R}_{+}^{m}\text{ is an }\mathcal{F}_{t}\text{%
-adapted} \\
\text{and }\left[ \mathbb{E}\left( \int_{0}^{T}\left\vert u\left( s\right)
\right\vert ^{2}\mathrm{d}s\right) ^{\frac{p}{2}}\right] ^{1\wedge \frac{1}{p%
}}<\infty%
\end{array}%
\right. \text{ }\right\} .
\end{equation*}

\item For vectors $\alpha ,\beta \in \mathbb{R}^{n}$, $\alpha \centerdot
\beta =\sum_{i=1}^{n}\alpha _{i}\beta _{i}$. For the matrix $%
M=(m_{ij})_{1\leq i,j\leq n}\in \mathbb{R}^{n\times n},$ Tr$\left( M\right)
=\sum_{i=1}^{n}m_{ii}$, the transpose of $M$ is $M^{\top }$. For any real
number we define $x^{+}:=$max$\{x,0\}$ and $x^{-}:=$max$\{-x,0\}$.

\item For any normed space $\left( E,\left\vert \cdot \right\vert \right) $,
$\mathcal{P}(E)$ is the set of all probability measures on $\left(
E,\left\vert \cdot \right\vert \right) $. For any $p\in N$, $\mathcal{P}%
_{p}(E)$ is the set of all probability measures of $p$th order on $\left(
E,\left\vert \cdot \right\vert \right) $, define by
\begin{equation*}
\mathcal{P}_{p}(E)=\left\{ \mu \in \mathcal{P}(E)\left\vert \left\Vert \mu
\right\Vert _{p}=\left( \int_{E}\left\vert x\right\vert ^{p}\mu \left(
\mathrm{d}x\right) \right) ^{\frac{1}{p}}<\infty \right. \right\} .
\end{equation*}%
For any probability measures $\mu ,\mu ^{\prime }$ in $\mathcal{P}_{p}(E)$,
the $p$th order Wasserstein distance on $\mathcal{P}_{p}(E)$ is defined as
\begin{equation*}
\mathcal{W}_{p}\left( \mu ,\mu ^{\prime }\right) =\inf_{\pi }\left(
\int_{E\times E}\left\vert z-z^{\prime }\right\vert ^{p}\pi \left(
dz,dz^{\prime }\right) \right) ^{\frac{1}{p}},
\end{equation*}%
where $\pi $ is a coupling of $\mu $ and $\mu ^{\prime }$ in the sense that $%
\pi \in \mathcal{P}(E\times E)$ with marginals $\mu $ and $\mu ^{\prime }$. $%
L^{2}(E)$ is the space of all square-integrable functions $f:E\rightarrow
\mathbb{R}$.

\item Let $L^{2}(\mathcal{F};E)$ as the space of all $E$-valued square
integrable random variables on $(\Omega ,\mathcal{F},P)$; for any $\varphi
\in L^{2}(\mathcal{F};E)$, we set $\left\Vert \varphi \right\Vert _{L^{2}}=%
\mathbb{E}\left[ \left\vert \varphi \right\vert ^{2}\right] ^{\frac{1}{2}}$.
We assume that the above sub-$\sigma $-field $\mathcal{G}$ of $\mathcal{F}$
which is independent of $\mathcal{F}_{\infty }$ and \textquotedblleft rich
enough\textquotedblright\ in the following sense:%
\begin{equation}
\mathcal{P}_{2}\left( \mathbb{R}^{n}\right) =\left\{ \mathbb{P}_{\xi
}\left\vert \xi \in L^{2}(\mathcal{G};E)\right. \right\} ,  \label{rich}
\end{equation}%
where $\mathbb{P}_{\xi }$ or $\mathcal{L}\left( \xi \right) $ denotes the
law of $\xi .$ From Lemma 2.1 in \cite{cg2022}, $\mathcal{G}$ is
\textquotedblleft rich enough\textquotedblright\ if and only if there exists
a $\mathcal{G}$-measurable random variable $U^{\mathcal{G}}:\Omega
\rightarrow \mathbb{R}$ having uniform distribution on $\left[ 0,1\right] $.
Particularly, if the probability space $\left( \Omega ,\mathcal{G},\mathbb{P}%
\right) $ is \emph{atomless} (namely, for any $A\in \mathcal{G}$ such that $%
\mathbb{P}\left( A\right) >0$ there exists $B\in \mathcal{G}$, $B\subset A$,
such that $0<\mathbb{P}\left( B\right) <\mathbb{P}\left( A\right) $), then
these two mentioned properties holds. (see page 352, \cite{cd18}).


\item The Lions derivative of a functional $f$, introduced in \cite{lions06}%
, is defined through the lift of $f.$ Take any function $f:\mathcal{P}_{2}(%
\mathbb{R}^{n})\rightarrow \mathbb{R}$, and let $\tilde{f}:L^{2}(\mathcal{F};%
\mathbb{R}^{n})\rightarrow \mathbb{R}$ be a life of $f$ such that $\tilde{f}%
\left( \xi \right) =f\left( \mathbb{P}_{\xi }\right) $ for any $\xi \in
L^{2}(\mathcal{F};\mathbb{R}^{n}).$ Then, $\tilde{f}$ is differentiable in
the Fr\'{e}chet sense at $\xi _{0}$ if there exists a linear continuous
mapping $D\tilde{f}\left( \xi _{0}\right) :L^{2}(\mathcal{F};\mathbb{R}%
^{n})\rightarrow \mathbb{R}$ such that%
\begin{equation*}
\tilde{f}\left( \xi \right) -\tilde{f}\left( \xi _{0}\right) =\mathbb{E}%
\left[ D\tilde{f}\left( \xi _{0}\right) \centerdot \left( \xi -\xi
_{0}\right) \right] +o\left( \left\Vert \xi -\xi _{0}\right\Vert
_{L^{2}}\right) ,
\end{equation*}%
as $\left\Vert \xi -\xi _{0}\right\Vert _{L^{2}}\rightarrow 0.$Whenever $%
\tilde{f}$ is the lift of a function $\tilde{f}$ in $\mathcal{P}_{2}(\mathbb{%
R}^{n})$, the law of $D\tilde{f}\left( \xi _{0}\right) $ depends on $\xi
_{0} $ only via its law $\mathbb{P}_{\xi _{0}}$, and
\begin{equation}
D\tilde{f}\left( \xi _{0}\right) =g_{0}\left( \xi _{0}\right) ,  \label{lg}
\end{equation}%
for some Borel function $g_{0}:\mathbb{R}^{n}\rightarrow \mathbb{R}^{n}$ .
(See e.g., \cite{cd18}, Chapter 5 and \cite{gt19}). The Lions derivative is
thus well defined.
\end{itemize}

\begin{definition}
\label{d1}We say $f$ is differentiable at $\mu _{0}=\mathbb{P}_{\xi _{0}}\in
\mathcal{P}_{2}(\mathbb{R}^{n})$ if its lift function $\tilde{f}$ is Fr\'{e}%
chet differentiable at $\xi _{0}$; and in this case, the function $g_{0}$ in
(\ref{lg}) is called the Lions derivative of $f$ at $\mu _{0}$, and denoted
as $\partial _{\mu }f\left( \mu _{0},\cdot \right) $.
\end{definition}

\begin{definition}
\label{d2}We say a function $f\in C^{1,1}\left( \mathcal{P}_{2}(\mathbb{R}%
^{n})\right) $, if its lift $\tilde{f}$ is Fr\'{e}chet differentiable, and
if there exists a continuous version of $\partial _{\mu }f\left( \mu
,x\right) $ such that (i) the mapping $\left( \mu ,x\right) \rightarrow
\partial _{\mu }f\left( \mu ,x\right) $ is jointly continuous with respect
to $\left( \mu ,x\right) $ and there is a constant $C>0$ such that
\begin{equation}
\left\vert \partial _{\mu }f\left( \mu ,x\right) \right\vert <C,  \label{l1}
\end{equation}%
for any $\mu \in \mathcal{P}_{2}(\mathbb{R}^{n})$ and any $x\in \mathbb{R}%
^{n}$; (ii) For any $\mu \in \mathcal{P}_{2}(\mathbb{R}^{n})$, the mapping $%
x\rightarrow \partial _{\mu }f\left( \mu ,x\right) $ is continuously
differentiable. Its derivative, denoted by $\partial _{x}\partial _{\mu
}f\left( \mu ,x\right) $, is jointly continuous with respect to $\left( \mu
,x\right) $; and there is a constant $C>0$ such that
\begin{equation}
\left\vert \partial _{x}\partial _{\mu }f\left( \mu ,x\right) \right\vert <C
\label{l2}
\end{equation}%
for any $\mu \in \mathcal{P}_{2}(\mathbb{R}^{n})$ and any $x\in \mathbb{R}%
^{n}$.
\end{definition}


\section{Constrained McKean-Vlasov stochastic LQ problem}

\label{sect3}

It is important to highlight two distinguishing features that set this
problem apart from conventional mean field LQ problems. First, the running
cost in this case may be identically zero, which means the problem falls
into the category of indefinite stochastic LQ control. Second, and more
significantly, the control is subject to constraints, a feature that
introduces the primary technical challenge. As a result, classical solution
methods for unconstrained LQ problems, such as the "completion of squares"
technique and DPP), both of which rely on the associated Riccati equation,
are no longer applicable.

Consider the stochastic controlled systems:
\begin{equation}
\left\{
\begin{array}{lll}
\mathrm{d}X\left( s\right) & = & \left[ AX_{s}+\bar{A}\mathbb{E}X\left(
s\right) +Bu\left( s\right) +b_{0}\right] \mathrm{d}s+\sum_{j=1}^{m}D_{j}u%
\left( s\right) \mathrm{d}W^{j}\left( s\right) , \\
X\left( t\right) & = & \xi \in L^{2}(\mathcal{G\vee F}_{t};\mathbb{R})%
\end{array}%
\right.  \label{sde1}
\end{equation}%
with cost functional defined by%
\begin{eqnarray}
J\left( u\left( \cdot \right) \right) &=&\mathbb{E}\bigg [%
G_{1}X_{T}^{2}+G_{2}\left( \mathbb{E}X_{T}\right) ^{2}+G_{3}X_{T}  \notag \\
&&+\int_{t}^{T}\left( Q_{1}X_{s}^{2}+Q_{2}\left( \mathbb{E}X_{s}\right)
^{2}+Q_{3}X_{s}\right) \mathrm{d}s\bigg ].  \label{cost1}
\end{eqnarray}%
We now impose the following assumption to enure the well-defined our control
problem.

\begin{enumerate}
\item[\textbf{(A1)}] Assume that $Q_{1}>0,$ $Q_{1}+Q_{2}\geq 0,$ $Q_{3}\leq
0,$ $G_{1}>0,$ $G_{1}+G_{2}\geq 0$, $A>0,$ $\bar{A}>0$ and $b_{0}$ are
scalars, $B^{\top }\in \mathbb{R}_{+}^{m},D_{j}^{\top }\in \mathbb{R}^{m}$ $%
\left( j=1,\ldots ,m\right) $ are column vectors. The matrix $%
\sum_{j=1}^{m}D_{j}^{\top }D_{j}$ is postulated to be non-singular.
\end{enumerate}

As usual, our goal is to minimize the cost functional over the admissible
control set and define the value function as follows:%
\begin{equation}
V\left( t,\xi \right) =\inf_{u\in \mathcal{U}_{ad}^{2,+}\left[ 0,T\right]
}J\left( t,\xi ,u\right) .  \label{v1}
\end{equation}%
Under assumption (H1), the value function in (4.3) is law-invariant (cf.
\cite{cg2022}). Moreover, for $\mu =\mathbb{P}_{\xi }\in \mathcal{P}%
_{2}\left( \mathbb{R}\right) ,$ with a slight abuse of notation, write%
\begin{equation*}
V\left( t,\mu \right) =V\left( t,\xi \right) =\inf_{u\in \mathcal{U}%
_{ad}^{2,+}\left[ 0,T\right] }J\left( t,\xi ,u\right) .
\end{equation*}%
The multivariate linear McKean-Vlasov controlled dynamics with coefficients
presented by%
\begin{eqnarray*}
b\left( x,\mu ,u\right) &=&Ax+\bar{A}\bar{\mu}+Bu+b_{0}, \\
\sigma \left( x,\mu ,u\right) &=&Du, \\
\Phi \left( x,\mu \right) &=&G_{1}x^{2}+G_{2}\bar{\mu}^{2}+G_{3}x, \\
f\left( x,\mu ,u\right) &=&Q_{1}x^{2}+Q_{2}\bar{\mu}^{2}+Q_{3}x,
\end{eqnarray*}%
where $\bar{\mu}=\int_{\mathbb{R}^{n}}x\mu \left( \mathrm{d}x\right) .$
Define an operator $H$ on $\mathbb{R}\times U\times \mathcal{P}_{2}\left(
\mathbb{R}\right) \times \mathbb{R}\times \mathbb{R}$, such that%
\begin{equation}
H\left( x,u,\mu ,p,Q\right) =f\left( x,u,\mu \right) +b\left( x,u,\mu
\right) p+\frac{1}{2}\text{\emph{Tr}}\left( \sigma \left( x,u,\mu \right)
^{\top }Q\right) .  \label{ha}
\end{equation}%
Clearly, the classical dynamic programming approach fails in this situation.
It is scarcely possible to establish the DPP due to the presence of
expectations in the coefficients and the nonlinear dependence on them.
Therefore, problems of this type indeed belong to a time-inconsistent
setting. Nevertheless, it should be noted that, although the problem itself
is time-inconsistent, it is possible to recover a certain form of the DPP by
extending the state space.

We now introduce the following HJB equation (derived from (4.6) in \cite%
{gpw23}), for any $\left( t,\mu \right) \in \left[ 0,T\right] \times
\mathcal{P}_{2}\left( \mathbb{R}\right) ,$%
\begin{equation}
\left\{
\begin{array}{l}
\partial _{t}\mathcal{V}\left( t,\mu \right) +\mathbb{E}\left[ \inf_{u\in
\mathbb{R}_{+}^{m}}H\left( \xi ,u,\mu ,\partial _{\mu }\mathcal{V}\left(
t,\mu ,\xi \right) ,\partial _{x}\partial _{\mu }\mathcal{V}\left( t,\mu
,\xi \right) \sigma \left( \xi ,u,\mu \right) \right) \right] =0, \\
\mathcal{V}\left( T,\mu \right) =\left\langle \Phi \left( \cdot ,\mu \right)
,\mu \right\rangle ,%
\end{array}%
\right.  \label{hjb1}
\end{equation}%
where $\left\langle \Phi \left( \cdot ,\mu \right) ,\mu \right\rangle =\int_{%
\mathbb{R}}\Phi \left( x,\mu \right) \mu \left( \mathrm{d}x\right) .$


In order to seek the optimal control, we need some results from optimization
theory.

Set
\begin{equation}
\bar{z}:=\underset{z\in \lbrack 0,\infty )^{m}}{\mathop{\rm argmin}\limits}%
\frac{1}{2}\left\Vert \left( D^{\prime }\right) ^{-1}\bar{z}+\left(
D^{\prime }\right) ^{-1}B^{\prime }\right\Vert ^{2}  \label{qup}
\end{equation}%
and
\begin{equation}
\bar{\vartheta}:=\left( D^{\prime }\right) ^{-1}\bar{z}+\left( D^{\prime
}\right) ^{-1}B.  \label{kesai}
\end{equation}%
Note that $\bar{\vartheta}$ is a column vector \textit{independent} of $x$.
We will use $\bar{\vartheta}$ to express the optimal feedback control later.


For any constant $\Gamma ,$ we set
\begin{equation*}
\bar{\mu}_{2}\left( \Gamma \right) =\Gamma \int_{\mathbb{R}^{d}}x^{2}\mu
\left( \mathrm{d}x\right) ,\text{ Var}\left( \mu \right) \left( \Gamma
\right) =\bar{\mu}_{2}\left( \Gamma \right) -\bar{\mu}^{2}\Gamma ,\text{ }%
\mu \in \mathcal{P}_{2}\left( \mathbb{R}\right) .
\end{equation*}%
Now let us consider the above McKean-Vlasov LQ problem (\ref{sde1})--(\ref%
{v1}).

We are going to find a value function $\mathcal{V}\left( t,\mu \right) ,$ $%
\mu =\mathbb{P}_{\xi }$ of the following type:%
\begin{equation*}
\mathcal{V}\left( t,\mu \right) =\text{Var}\left( \mu \right) \left(
P_{1}\left( t\right) \right) +P_{2}\left( t\right) \bar{\mu}^{2}+P_{3}\left(
t\right) \bar{\mu}+P_{4}\left( t\right) ,
\end{equation*}%
where $P_{1},P_{2}\in C^{1}\left( \left[ 0,T\right] ;\mathbb{R}\right) ,$ $%
P_{3}\in C^{1}\left( \left[ 0,T\right] ;\mathbb{R}\right) $ and $P_{4}\in
C^{1}\left( \left[ 0,T\right] ;\mathbb{R}\right) .$ It is easy to compute%
\begin{eqnarray*}
\partial _{t}\mathcal{V}\left( t,\mu \right) &=&\text{Var}\left( \mu \right)
\left( P_{1}^{\prime }\left( t\right) \right) +P_{2}^{\prime }\left(
t\right) \bar{\mu}^{2}+P_{3}^{\prime }\left( t\right) \bar{\mu}%
+P_{4}^{\prime }\left( t\right) , \\
\partial _{\mu }\mathcal{V}\left( t,\mu \right) \left( x\right)
&=&2P_{1}\left( t\right) \left( x-\bar{\mu}\right) +2P_{2}\left( t\right)
\bar{\mu}+P_{3}\left( t\right) , \\
\partial _{x}\partial _{\mu }\mathcal{V}\left( t,\mu \right) \left( x\right)
&=&2P_{1}\left( t\right) .
\end{eqnarray*}%
Now we are ready to derive the $P_{1},P_{2},P_{3}$ and $P_{4}$ according to
the HJB equation (\ref{hjb1}).

First, we compare the terms in
\begin{eqnarray*}
\mathcal{V}\left( T,\mu \right) &=&\text{Var}\left( \mu \right) \left(
P_{1}\left( T\right) \right) +P_{2}\left( T\right) \bar{\mu}^{2}+P_{3}\left(
T\right) \bar{\mu}+P_{4}\left( T\right) \\
&=&\text{Var}\left( \mu \right) \left( G_{1}\right) +\left(
G_{1}+G_{2}\right) \bar{\mu}^{2}+G_{3}\bar{\mu},
\end{eqnarray*}%
which implies that
\begin{equation*}
P_{1}\left( T\right) =G_{1},\text{ }P_{2}\left( T\right) =G_{1}+G_{2},\text{
}P_{3}\left( T\right) =G_{3},\text{ }P_{4}\left( T\right) =0.
\end{equation*}%
Meanwhile
\begin{eqnarray}
&&\partial _{t}\mathcal{V}\left( t,\mu \right) +\mathbb{E}\left[ \inf_{u\in
\mathbb{R}_{+}^{m}}H\left( \xi ,u,\mu ,\partial _{\mu }\mathcal{V}\left(
t,\mu ,\xi \right) ,\partial _{x}\partial _{\mu }\mathcal{V}\left( t,\mu
,\xi \right) \sigma \left( \xi ,u,\mu \right) \right) \right]  \notag \\
&=&\text{Var}\left( \mu \right) \left( P_{1}^{\prime }\left( t\right)
\right) +P_{2}^{\prime }\left( t\right) \bar{\mu}^{2}+P_{3}^{\prime }\left(
t\right) \bar{\mu}+P_{4}^{\prime }\left( t\right)  \notag \\
&&+\mathbb{E}\Big [\inf_{u\in \mathbb{R}_{+}^{m}}\Big (\left( A\xi +\bar{A}%
\bar{\mu}+Bu+b_{0}\right) \centerdot \left( 2P_{1}\left( t\right) \left( \xi
-\bar{\mu}\right) +2P_{2}\left( t\right) \bar{\mu}+P_{3}\left( t\right)
\right)  \notag \\
&&+\text{Tr}\left( u^{\top }D^{\top }P_{1}\left( t\right) Du\right)
+Q_{1}\xi \centerdot \xi +Q_{2}\bar{\mu}^{2}+Q_{3}\centerdot \xi \Big )\Big ]
\notag \\
&=&\text{Var}\left( \mu \right) \left( P_{1}^{\prime }\left( t\right)
\right) +P_{2}^{\prime }\left( t\right) \bar{\mu}^{2}+P_{3}^{\prime }\left(
t\right) \bar{\mu}+P_{4}^{\prime }\left( t\right)  \notag \\
&&+\mathbb{E}\Bigg \{\left( A\xi +\bar{A}\bar{\mu}+b_{0}\right) \centerdot
\left( 2P_{1}\left( t\right) \left( \xi -\bar{\mu}\right) +2P_{2}\left(
t\right) \bar{\mu}+P_{3}\left( t\right) \right)  \notag \\
&&+Q_{1}\xi ^{2}+Q_{2}\bar{\mu}^{2}+Q_{3}\xi  \notag \\
&&+2\inf_{u\in \mathbb{R}_{+}^{m}}\left[ Bu\centerdot \left( P_{1}\left(
t\right) \left( \xi -\bar{\mu}\right) +P_{2}\left( t\right) \bar{\mu}+\frac{1%
}{2}P_{3}\left( t\right) \right) +\frac{1}{2}u^{\top }D^{\top }P_{1}\left(
t\right) Du\right] \Bigg \}  \notag \\
&=&\text{Var}\left( \mu \right) \left( P_{1}^{\prime }\left( t\right)
+2AP_{1}\left( t\right) +Q_{1}\right)  \notag \\
&&+\left[ P_{2}^{\prime }\left( t\right) +2\left( A+\bar{A}\right)
P_{2}\left( t\right) +\left( Q_{1}+Q_{2}\right) \right] \bar{\mu}^{2}  \notag
\\
&&+\left[ P_{3}^{\prime }\left( t\right) +\left( A+\bar{A}\right)
P_{3}\left( t\right) +2P_{2}\left( t\right) b_{0}+Q_{3}\right] \bar{\mu}
\notag \\
&&+P_{4}^{\prime }\left( t\right) +b_{0}P_{3}\left( t\right)  \notag \\
&&+2P_{1}\left( t\right) \mathbb{E}\left\{ \inf_{u\in \mathbb{R}_{+}^{m}}%
\left[ \frac{1}{2}u^{\top }D^{\top }Du+\left( \xi -\bar{\mu}+\frac{%
2P_{2}\left( t\right) \bar{\mu}+P_{3}\left( t\right) }{2P_{1}\left( t\right)
}\right) Bu\right] \right\} ,  \label{min}
\end{eqnarray}%
where $D^{\top }=\left( D_{1}^{\top },\ldots ,D_{m}^{\top }\right) .$

By Lemma \ref{lem:s1} with
\begin{equation*}
\alpha =-\left[ \xi -\bar{\mu}+\frac{2P_{2}\left( t\right) \bar{\mu}%
+P_{3}\left( t\right) }{2P_{1}\left( t\right) }\right] >0\text{ with }\mu =%
\mathbb{P}_{\xi },
\end{equation*}%
it follows that the minimizer of (\ref{min}) is achieved by%
\begin{equation}
u^{\ast }\left( t,\xi ,\mu \right) =-D^{-1}\bar{\vartheta}\cdot \frac{\xi -%
\bar{\mu}+2P_{2}\left( t\right) \bar{\mu}+P_{3}\left( t\right) }{%
2P_{1}\left( t\right) },\text{ with }\mu =\mathbb{P}_{\xi },  \label{opti1}
\end{equation}%
where $\bar{\vartheta}$ is defined in (\ref{kesai}). We now define the
region $\Pi _{1}$ w.r.t. $\left( t,x,\mu \right) $ as%
\begin{equation}
\Pi _{1}=\left\{ \left( t,x,\mu \right) \in \left[ 0,T\right] \times \mathbb{%
R}\times \mathcal{P}_{2}\left( \mathbb{R}\right) \left\vert x-\bar{\mu}+%
\frac{2P_{2}\left( t\right) \bar{\mu}+P_{3}\left( t\right) }{2P_{1}\left(
t\right) }<0\right. \right\} .  \label{R1}
\end{equation}%
Hence, on $\Pi _{1},$ the value function admits
\begin{equation*}
\mathcal{V}\left( t,\mu \right) =\mathcal{V}_{1}\left( t,\mu \right) =\text{%
Var}\left( \mu \right) \left( P_{1}\left( t\right) \right) +P_{2}\left(
t\right) \bar{\mu}^{2}+P_{3}\left( t\right) \bar{\mu}+P_{4}\left( t\right) .
\end{equation*}%
By Theorem 4.1 in \cite{gpw23}, we conclude that $u^{\ast }\left( t,x,\mu
\right) $ defined above is an optimal control for $\mathcal{V}\left( t,\mu
\right) $ on $\Pi _{1}.$

Whilst
\begin{eqnarray}
&&\mathbb{E}\left\{ \inf_{u\in \mathbb{R}_{+}^{m}}\left[ \frac{1}{2}u^{\top
}D^{\top }Du+\left( \xi -\bar{\mu}+\frac{2P_{2}\left( t\right) \bar{\mu}%
+P_{3}\left( t\right) }{2P_{1}\left( t\right) }\right) Bu\right] \right\}
\label{inf} \\
&=&-\frac{1}{2}\mathbb{E}\left[ \left( \xi -\bar{\mu}+\frac{2P_{2}\left(
t\right) \bar{\mu}+P_{3}\left( t\right) }{2P_{1}\left( t\right) }\right)
^{2}\cdot \left\Vert \bar{\vartheta}\right\Vert ^{2}\right]  \notag \\
&=&-\frac{1}{2}\left[ \text{Var}\left( \mu \right) +\frac{P_{2}^{2}\left(
t\right) }{P_{1}^{2}\left( t\right) }\bar{\mu}^{2}+\frac{P_{2}\left(
t\right) P_{3}\left( t\right) }{P_{1}^{2}\left( t\right) }\bar{\mu}+\frac{1}{%
4}\frac{P_{3}^{2}\left( t\right) }{P_{1}^{2}\left( t\right) }\right] \cdot
\left\Vert \bar{\vartheta}\right\Vert ^{2}.  \notag
\end{eqnarray}%
Substituting $u^{\ast }\left( t,x,\mu \right) $ back into (\ref{min}), we
are able to rewrite (\ref{min}) as follows:%
\begin{eqnarray}
0 &=&\partial _{t}\mathcal{V}\left( t,\mu \right) +\inf_{u\in \mathbb{R}%
_{+}^{m}}\mathbb{E}\left[ H\left( \xi ,u,\mu ,\partial _{\mu }\mathcal{V}%
\left( t,\mu ,\xi \right) ,\partial _{x}\partial _{\mu }\mathcal{V}\left(
t,\mu ,\xi \right) \sigma \left( \xi ,u,\mu \right) \right) \right]  \notag
\\
&=&\text{Var}\left( \mu \right) \left( \dot{P}_{1}\left( t\right)
+2AP_{1}\left( t\right) +Q_{1}-\left\Vert \bar{\vartheta}\right\Vert
^{2}\right)  \notag \\
&&+\left[ \dot{P}_{2}\left( t\right) +2\left( A+\bar{A}\right) P_{2}\left(
t\right) -\frac{P_{2}^{2}\left( t\right) }{P_{1}\left( t\right) }\left\Vert
\bar{\vartheta}\right\Vert ^{2}+Q_{1}+Q_{2}\right] \bar{\mu}^{2}  \notag \\
&&+\left[ \dot{P}_{3}\left( t\right) +\left( \bar{A}+A\right) P_{3}\left(
t\right) +2P_{2}\left( t\right) b_{0}+Q_{3}-\frac{P_{2}\left( t\right)
P_{3}\left( t\right) }{P_{1}\left( t\right) }\left\Vert \bar{\vartheta}%
\right\Vert ^{2}\right] \bar{\mu}  \notag \\
&&+\dot{P}_{4}\left( t\right) +b_{0}P_{3}\left( t\right) -\frac{1}{4}\frac{%
P_{3}^{2}\left( t\right) }{P_{1}\left( t\right) }\left\Vert \bar{\vartheta}%
\right\Vert ^{2}.  \label{min2}
\end{eqnarray}%
Now comparing terms in Var$\left( \mu \right) $, $\bar{\mu}^{2}$, $\bar{\mu}$
in (\ref{min2}), we obtain the following ODEs system for $P_{1}\left(
t\right) $, $P_{2}\left( t\right) $, $P_{3}\left( t\right) $ and $%
P_{4}\left( t\right) $,%
\begin{equation}
\left\{
\begin{array}{l}
\dot{P}_{1}\left( t\right) +2AP_{1}\left( t\right) +Q_{1}-\left\Vert \bar{%
\vartheta}\right\Vert ^{2}=0, \\
P_{1}\left( T\right) =G_{1},%
\end{array}%
\right.  \label{ode1}
\end{equation}%
\begin{equation}
\left\{
\begin{array}{l}
\dot{P}_{2}\left( t\right) -\frac{\left\Vert \bar{\vartheta}\right\Vert ^{2}%
}{P_{1}\left( t\right) }P_{2}^{2}\left( t\right) +2\left( A+\bar{A}\right)
P_{2}\left( t\right) +Q_{1}+Q_{2}=0, \\
P_{2}\left( T\right) =G_{1}+G_{2},%
\end{array}%
\right.  \label{ode2}
\end{equation}%
\begin{equation}
\left\{
\begin{array}{l}
\dot{P}_{3}\left( t\right) +\left( \bar{A}+A-\frac{P_{2}\left( t\right) }{%
P_{1}\left( t\right) }\left\Vert \bar{\vartheta}\right\Vert ^{2}\right)
P_{3}\left( t\right) +2P_{2}\left( t\right) b_{0}+Q_{3}=0, \\
P_{3}\left( T\right) =G_{3},%
\end{array}%
\right.  \label{ode3}
\end{equation}%
and
\begin{equation}
\left\{
\begin{array}{l}
\dot{P}_{4}\left( t\right) +b_{0}P_{3}\left( t\right) -\frac{1}{4}\frac{%
P_{3}^{2}\left( t\right) }{P_{1}\left( t\right) }\left\Vert \bar{\vartheta}%
\right\Vert ^{2}=0, \\
P_{4}\left( T\right) =0.%
\end{array}%
\right.  \label{ode4}
\end{equation}

\begin{remark}
Clearly, $P_{2}\left( \cdot \right) $ in (\ref{ode2}) is a classical Riccati
equation. Generally, it is impossible to get the explicit the expression of
solution to (\ref{ode2}). Therefore, unlike in \cite{LZL2002}, the analysis
of value functions becomes more difficult. Nonetheless, whenever, $%
Q_{1}+Q_{2}=G_{1}+G_{2}=0,$ immediately, $P_{2}\left( t\right) \equiv 0,$ $%
\forall t\in \left[ 0,T\right] ,$ which is corresponding to the variance
minimization problem (see Section \ref{sect4}).
\end{remark}

Next we proceed to the region $\Pi _{2}$ defined by%
\begin{equation}
\Pi _{2}=\left\{ \left( t,x,\mu \right) \in \left[ 0,T\right] \times \mathbb{%
R}\times \mathcal{P}_{2}\left( \mathbb{R}\right) \left\vert x-\bar{\mu}+%
\frac{2P_{2}\left( t\right) \bar{\mu}+P_{3}\left( t\right) }{2P_{1}\left(
t\right) }>0\right. \right\} .  \label{R2}
\end{equation}%
Analogous to the derivations for the previous case, particularly, the unique
minimizer of (\ref{inf}) equals to zero, we obtain%
\begin{equation}
\left\{
\begin{array}{l}
\dot{\widetilde{P}}_{1}\left( t\right) +2A\widetilde{P}_{1}\left( t\right)
+Q_{1}=0, \\
\widetilde{P}\left( T\right) =G_{1},%
\end{array}%
\right.  \label{od1}
\end{equation}%
\begin{equation}
\left\{
\begin{array}{l}
\dot{\widetilde{P}}_{2}\left( t\right) +2\left( A+\bar{A}\right) \widetilde{P%
}_{2}\left( t\right) +Q_{1}+Q_{2}=0, \\
\widetilde{P}_{2}\left( T\right) =G_{1}+G_{2},%
\end{array}%
\right.  \label{od2}
\end{equation}%
\begin{equation}
\left\{
\begin{array}{l}
\dot{\widetilde{P}}_{3}\left( t\right) +\left( \bar{A}+A\right) \widetilde{P}%
_{3}\left( t\right) +2\widetilde{P}_{2}\left( t\right) b_{0}+Q_{3}=0, \\
\widetilde{P}_{3}\left( T\right) =G_{3},%
\end{array}%
\right.  \label{od3}
\end{equation}%
and
\begin{equation}
\left\{
\begin{array}{l}
\dot{\widetilde{P}}_{4}\left( t\right) +b_{0}\widetilde{P}_{3}\left(
t\right) =0, \\
\widetilde{P}_{4}\left( T\right) =0.%
\end{array}%
\right.  \label{od4}
\end{equation}%
Hence, on $\Pi _{2},$ the value function reads
\begin{equation*}
\mathcal{V}\left( t,\mu \right) =\mathcal{V}_{2}\left( t,\mu \right) =\text{%
Var}\left( \mu \right) \left( \tilde{P}_{1}\left( t\right) \right) +\tilde{P}%
_{2}\left( t\right) \bar{\mu}^{2}+\tilde{P}_{3}\left( t\right) \bar{\mu}+%
\tilde{P}_{4}\left( t\right) .
\end{equation*}%
Applying Theorem 4.1 in \cite{gpw23} again, clearly $u^{\ast }\left( t,\mu
\right) =0$ is an optimal control for $\mathcal{V}\left( t,\mu \right) $ on $%
\Pi _{2}.$

\begin{remark}
Note that, in contrast to the result presented in \cite{LZL2002}, the
optimal control $u^{\ast }\left( \cdot \right) $ in (\ref{opti1}) depends
not only on the parameter $\bar{\vartheta}$ but also on the probability
measure $\mu $ and $x$. Note that $\bar{\vartheta}$ does not depend on $x$.
This means that $P_{i}(t),$ $i=1,\ldots ,4$, which also depend on $t$, do
not depend on $x$ and $\mu $. Hence, the expressions for $\mathcal{V}%
_{t}(t,\mu )$, $\mathcal{V}_{\mu }(t,\mu )$ and $\partial _{x}\partial _{\mu
}\mathcal{V}(t,\mu )$ do not involve terms of the form $P_{i}(t),$ $%
i=1,\ldots ,4$, etc. Due to this the closed form expressions for the value
function can be obtained.
\end{remark}

It is necessary to point out that the region $\Pi _{2}$ depends on $%
P_{i},i=1,\ldots 4.$ Note that however, $\widetilde{P}_{2}\left( t\right) $
is a normal ODE, while $P_{2}$ is a classical Riccati equation (impossible
to get the explicit solution).

\begin{remark}
\label{keyrem}To interpret the roles of $Q_{i},i=1,\ldots ,3,$ in $%
P_{i},i=1,\ldots 4,$ we start with (\ref{ode2}). Defining $\eta \left(
t\right) =\frac{P_{3}\left( t\right) }{P_{2}\left( t\right) }$ and $\tilde{%
\eta}\left( t\right) =\frac{\tilde{P}_{3}\left( t\right) }{\tilde{P}%
_{2}\left( t\right) },$ it follows from (\ref{ode1}) and (\ref{ode2}) that
\begin{eqnarray*}
\dot{\eta}\left( t\right) &=&\frac{P_{2}\left( t\right) \dot{P}_{3}\left(
t\right) -P_{3}\left( t\right) \dot{P}_{2}\left( t\right) }{P_{2}^{2}\left(
t\right) } \\
&=&\frac{P_{2}\left( t\right) \left[ -\left( \bar{A}+A-\frac{P_{2}\left(
t\right) }{P_{1}\left( t\right) }\left\Vert \bar{\vartheta}\right\Vert
^{2}\right) P_{3}\left( t\right) -2P_{2}\left( t\right) b_{0}-Q_{3}\right] }{%
P_{2}^{2}\left( t\right) } \\
&&-\frac{P_{3}\left( t\right) \left[ \frac{\left\Vert \bar{\vartheta}%
\right\Vert ^{2}}{P_{1}\left( t\right) }P_{2}^{2}\left( t\right) -2\left( A+%
\bar{A}\right) P_{2}\left( t\right) -Q_{1}-Q_{2}\right] }{P_{2}^{2}\left(
t\right) } \\
&=&\frac{\left( A+\bar{A}\right) P_{2}\left( t\right) P_{3}\left( t\right)
-2b_{0}P_{2}^{2}\left( t\right) -P_{2}\left( t\right) Q_{3}+P_{3}\left(
t\right) \left( Q_{1}+Q_{2}\right) }{P_{2}^{2}\left( t\right) } \\
&=&\left( A+\bar{A}\right) \eta \left( t\right) -2b_{0}+\frac{Q_{1}+Q_{2}}{%
P_{2}\left( t\right) }-\frac{Q_{3}}{P_{2}\left( t\right) }.
\end{eqnarray*}%
Similarly,
\begin{equation*}
\dot{\tilde{\eta}}\left( t\right) =\left( A+\bar{A}\right) \tilde{\eta}%
\left( t\right) -2b_{0}+\frac{Q_{1}+Q_{2}}{\tilde{P}_{2}\left( t\right) }-%
\frac{Q_{3}}{\tilde{P}_{2}\left( t\right) }.
\end{equation*}%
But generally $P_{2}\left( t\right) \neq \tilde{P}_{2}\left( t\right) ,$
which immediately implies that, in general, $\frac{P_{3}\left( t\right) }{%
P_{2}\left( t\right) }\neq \frac{\tilde{P}_{3}\left( t\right) }{\tilde{P}%
_{2}\left( t\right) }.$ Particularly, if we suppose that
\begin{equation}
Q_{3}=Q_{1}+Q_{2}=0.  \label{Q}
\end{equation}%
Then $\eta \left( t\right) =\tilde{\eta}\left( t\right) ,$ namely,
\begin{eqnarray}
\eta \left( t\right) &=&\frac{P_{3}\left( t\right) }{P_{2}\left( t\right) }
\notag \\
&=&\frac{\tilde{P}_{3}\left( t\right) }{\tilde{P}_{2}\left( t\right) }
\notag \\
&=&\frac{G_{3}}{G_{1}+G_{2}}\exp \left\{ -\left( A+\bar{A}\right) \left(
T-t\right) \right\}  \notag \\
&&+\frac{2b_{0}}{A+\bar{A}}\left( 1-\exp \left( -\left( A+\bar{A}\right)
\left( T-t\right) \right) \right) .  \label{yita}
\end{eqnarray}
\end{remark}

Now we consider the switching curve $\Pi _{3}$ defined by%
\begin{equation}
\Pi _{3}=\left\{ \left( t,x,\mu \right) \in \left[ 0,T\right] \times \mathbb{%
R}\times \mathcal{P}_{2}\left( \mathbb{R}\right) \left\vert x-\bar{\mu}+%
\frac{2P_{2}\left( t\right) \bar{\mu}+P_{3}\left( t\right) }{P_{1}\left(
t\right) }=0\right. \right\} ,  \label{R3}
\end{equation}%
where the \textit{discontinuous} of $\mathcal{V}$ may happen. According
Lemma \ref{lem:h}, we see that on $\Pi _{3}$ the unique minimizer $u^{\ast
}\left( t,\mu \right) =0.$

In addition, if $Q_{1}+Q_{2}=Q_{3}=0,$ from Remark \ref{keyrem}, it yields $%
\bar{\mu}=-\frac{\eta \left( t\right) }{2}$ and then
\begin{eqnarray*}
\mathcal{V}_{1}\left( t,\mu \right) &=&\text{Var}\left( \mu \right) \left(
P_{1}\left( t\right) \right) -\frac{\eta \left( t\right) ^{2}}{4}P_{2}\left(
t\right) +P_{4}\left( t\right) , \\
\mathcal{V}_{2}\left( t,\mu \right) &=&\text{Var}\left( \mu \right) \left(
\tilde{P}_{1}\left( t\right) \right) -\frac{\eta \left( t\right) ^{2}}{4}%
\tilde{P}_{2}\left( t\right) +\tilde{P}_{4}\left( t\right) .
\end{eqnarray*}

\begin{example}
Let $Q_{1}=Q_{2}=Q_{3}=b_{0}=0,G_{3}=-\beta \leq 0$ and $G_{2}=-G_{1}<0.$
Immediately,%
\begin{equation*}
\left\{
\begin{array}{l}
P_{1}\left( t\right) =G_{1}e^{2A\left( T-t\right) }-\frac{\left\Vert \bar{%
\vartheta}\right\Vert ^{2}}{2A}\left( e^{2A\left( T-t\right) }-1\right) , \\
P_{2}\left( t\right) =0, \\
P_{3}\left( t\right) =-\beta \exp \left( r\left( T-t\right) \right) , \\
P_{4}\left( t\right) =\frac{\beta ^{2}}{4}-\frac{\beta ^{2}}{4}\exp \left(
\left\Vert \bar{\theta}\right\Vert ^{2}\left( T-t\right) \right) ,%
\end{array}%
\right. \left\{
\begin{array}{l}
\tilde{P}_{1}\left( t\right) =G_{1}\exp \left( 2A\left( T-t\right) \right) ,
\\
\tilde{P}_{2}\left( t\right) =0, \\
\tilde{P}_{3}\left( t\right) =-\beta \exp \left( r\left( T-t\right) \right) ,
\\
\tilde{P}_{4}\left( t\right) =0.%
\end{array}%
\right.
\end{equation*}%
We observe that if $\beta >0,$ then for any $\mu \in \mathcal{P}_{2}\left(
\mathbb{R}\right) ,$ $2P_{2}\left( t\right) \bar{\mu}+P_{3}\left( t\right)
<0,$ therefore, $x-\bar{\mu}+\frac{2P_{2}\left( t\right) \bar{\mu}%
+P_{3}\left( t\right) }{2P_{1}\left( t\right) }>0$
\end{example}

\begin{example}
$\Pi _{2}=\mathbb{R}\times \mathcal{P}_{2}\left( \mathbb{R}\right) ,$ $%
\mathcal{V}\left( t,\mu \right) $ admits a unique smooth solution
\begin{eqnarray*}
\mathcal{V}\left( t,\mu \right) &=&\text{Var}\left( \mu \right) \left(
G_{1}\exp \left\{ 2A\left( T-t\right) \right\} \right) -\beta \exp \left(
r\left( T-t\right) \right) \bar{\mu} \\
&&+\frac{\beta ^{2}}{4}-\frac{\beta ^{2}}{4}\exp \left( \left\Vert \bar{%
\theta}\right\Vert ^{2}\left( T-t\right) \right) ,
\end{eqnarray*}%
which fortunately corresponding to mean-variance problem (see Section \ref%
{sect4} below); However, if $\beta =0,$ then $2P_{2}\left( t\right) \bar{\mu}%
+P_{3}\left( t\right) \equiv 0,$ which means $\Pi _{3}=\mathbb{R}\times
\mathcal{P}_{2}\left( \mathbb{R}\right) ,$ so
\begin{equation*}
\mathcal{V}\left( t,\mu \right) =\text{Var}\left( \mu \right) \left(
G_{1}\exp \left\{ 2A\left( T-t\right) \right\} \right) ,
\end{equation*}%
which is not trivial, since at time $t>0,$ a agent possesses a random wealth
$\xi ,$ then the investment risk can be captued by $G_{1}\exp \left\{
2A\left( T-t\right) \right\} \mathbb{D}\xi .$
\end{example}

In general, the value function $\mathcal{V}\left( t,\mu \right) $ might
\emph{not} be continuous on $\Pi _{3}$, which is completely different from $%
\Gamma _{3}$ defined in \cite{LZL2002}. Nevertheless, due to the complexity
of $P_{2}\left( \cdot \right) ,$ we are able to present a partial result
currently. Analyzing the relationship of size of the terms between $\left(
P_{1}\left( t\right) ,\tilde{P}_{1}\left( t\right) \right) ,\left(
P_{2}\left( t\right) ,\widetilde{P}_{2}\left( t\right) \right) $ and $\left(
P_{4}\left( t\right) ,\widetilde{P}_{4}\left( t\right) \right) $ is equally
important.

\begin{lemma}
On $\Pi _{3},$ under \emph{(A1)}, in addition, suppose that $Q_{1}\geq 0,$ $%
G_{3}\leq 0,$ $A\geq 0.$ Then
\begin{equation*}
0<P_{1}\left( t\right) =\widetilde{P}_{1}\left( t\right) ,P_{2}\left(
t\right) \leq \widetilde{P}_{2}\left( t\right) ,P_{4}\left( t\right) \leq
\widetilde{P}_{4}\left( t\right) ,
\end{equation*}
\end{lemma}

\begin{proof}
First after some basic derivation, we get
\begin{equation*}
P_{1}\left( t\right) =\widetilde{P}_{1}\left( t\right) =G_{1}e^{2A\left(
T-t\right) }+\frac{Q_{1}}{2A}\left( e^{2A\left( T-t\right) }-1\right) >0.
\end{equation*}%
Now we define $\Delta P_{2}\left( t\right) =P_{2}\left( t\right) -\tilde{P}%
_{2}\left( t\right) .$ Then
\begin{eqnarray*}
&&\left\vert \Delta P_{2}\left( t\right) ^{+}\right\vert ^{2} \\
&=&\int_{t}^{T}2\mathbf{I}_{\left\{ \Delta P_{2}\left( s\right) >0\right\}
}\Delta P_{2}\left( s\right) ^{+}\left[ -\frac{\left\Vert \bar{\vartheta}%
\right\Vert ^{2}}{P_{1}\left( s\right) }P_{2}^{2}\left( s\right) +2\left( A+%
\bar{A}\right) \Delta P_{2}\left( s\right) ^{+}\right] \mathrm{d}s \\
&\leq &\int_{t}^{T}4\left( A+\bar{A}\right) \mathbf{I}_{\left\{ \Delta
P_{2}\left( s\right) >0\right\} }\left\vert \Delta P_{2}\left( s\right)
^{+}\right\vert ^{2}\mathrm{d}s.
\end{eqnarray*}%
Immediately, from the backward Gronwall Bellman Lemma (see Lemma \ref{gron}
in Appendix), we have $\left\vert \Delta P_{2}\left( t\right)
^{+}\right\vert ^{2}\equiv 0,$ so $P_{2}\left( t\right) \leq \tilde{P}%
_{2}\left( t\right) ,$ $\forall t\in \left[ 0,T\right] .$ As for $%
P_{4}\left( t\right) ,$ we have%
\begin{eqnarray*}
P_{4}\left( t\right) &=&\left[ \frac{1}{4}\frac{P_{3}^{2}\left( t\right) }{%
P_{1}\left( t\right) }\left\Vert \bar{\vartheta}\right\Vert
^{2}-b_{0}P_{3}\left( t\right) \right] \left( t-T\right) , \\
\tilde{P}_{4}\left( t\right) &=&-b_{0}P_{3}\left( t\right) \left( t-T\right)
.
\end{eqnarray*}%
Clearly, $P_{4}\left( t\right) \leq \tilde{P}_{4}\left( t\right) ,$ $\forall
t\in \left[ 0,T\right] .$ The proof is thus complete.
\end{proof}

\begin{example}
Let us consider the case: $Q_{3}=Q_{1}+Q_{2}=0,$ $Q_{1}\geq 0,$ $b_{0}\geq
0,G_{3}\leq 0,A\geq 0.$ Apparently, the Riccati equation (\ref{ode2})
becomes a Bernoulli's equation. Therefore, due to $P_{1}\left( t\right) >0,$
one can get%
\begin{eqnarray*}
P_{1}\left( t\right) &=&G_{1}e^{2A\left( T-t\right) }+\frac{Q_{1}-\left\Vert
\bar{\vartheta}\right\Vert ^{2}}{2A}\left( e^{2A\left( T-t\right) }-1\right)
\\
\widetilde{P}_{1}\left( t\right) &=&G_{1}e^{2A\left( T-t\right) }+\frac{Q_{1}%
}{2A}\left( e^{2A\left( T-t\right) }-1\right) >0.
\end{eqnarray*}%
\begin{eqnarray*}
P_{2}\left( t\right) &=&\left( \frac{1}{G_{1}+G_{2}}e^{-2\left( A+\bar{A}%
\right) \left( T-t\right) }+\frac{\left\Vert \bar{\vartheta}\right\Vert ^{2}%
}{2\left( A+\bar{A}\right) P_{1}\left( t\right) }\left[ 1-e^{-2\left( A+\bar{%
A}\right) \left( T-t\right) }\right] \right) ^{-1}>0, \\
\tilde{P}_{2}\left( t\right) &=&\left( \frac{1}{G_{1}+G_{2}}e^{-2\left( A+%
\bar{A}\right) \left( T-t\right) }\right) ^{-1}>0,
\end{eqnarray*}%
\begin{eqnarray*}
P_{4}\left( t\right) &=&\int_{t}^{T}\left( b_{0}P_{3}\left( s\right) -\frac{1%
}{4}\frac{P_{3}^{2}\left( s\right) }{P_{1}\left( s\right) }\left\Vert \bar{%
\vartheta}\right\Vert ^{2}\right) \mathrm{d}s, \\
\tilde{P}_{4}\left( t\right) &=&\int_{t}^{T}b_{0}P_{3}\left( s\right)
\mathrm{d}s.
\end{eqnarray*}%
Now
\begin{eqnarray*}
\mathcal{V}_{1}\left( t,\mu \right) &=&\text{Var}\left( \mu \right) \left[
G_{1}e^{2A\left( T-t\right) }+\frac{Q_{1}}{2A}\left( e^{2A\left( T-t\right)
}-1\right) \right] \\
&&-\frac{\eta \left( t\right) ^{2}}{4}\left( \frac{1}{G_{1}+G_{2}}%
e^{-2\left( A+\bar{A}\right) \left( T-t\right) }+\frac{\left\Vert \bar{%
\vartheta}\right\Vert ^{2}}{2\left( A+\bar{A}\right) P_{1}\left( t\right) }%
\left[ 1-e^{-2\left( A+\bar{A}\right) \left( T-t\right) }\right] \right)
^{-1} \\
&&+\int_{t}^{T}\left( b_{0}P_{3}\left( s\right) -\frac{1}{4}\frac{%
P_{3}^{2}\left( s\right) }{P_{1}\left( s\right) }\left\Vert \bar{\vartheta}%
\right\Vert ^{2}\right) \mathrm{d}s, \\
\mathcal{V}_{2}\left( t,\mu \right) &=&\text{Var}\left( \mu \right) \left[
G_{1}e^{2A\left( T-t\right) }+\frac{Q_{1}}{2A}\left( e^{2A\left( T-t\right)
}-1\right) \right] \\
&&-\frac{\eta \left( t\right) ^{2}}{4}\left( \frac{1}{G_{1}+G_{2}}%
e^{-2\left( A+\bar{A}\right) \left( T-t\right) }\right)
^{-1}+\int_{t}^{T}b_{0}P_{3}\left( s\right) \mathrm{d}s,
\end{eqnarray*}%
where $\eta \left( t\right) $ is defined in (\ref{yita}). These expressions
make the analysis rather complicated.
\end{example}

We are now asserting a result in the following.

\begin{theorem}
Assume that \emph{(A1)} holds. Then the average optimal control of Problem (%
\ref{v1}) can be represented as
\begin{eqnarray}
u^{\ast }\left( t,x,\mu \right) &=&\left( u_{1}^{\ast }\left( t,x,\mu
\right) ,\ldots ,u_{m}^{\ast }\left( t,x,\mu \right) \right) ^{\top }  \notag
\\
&=&\left\{
\begin{array}{ll}
-D^{-1}\bar{\vartheta}\cdot \left( x-\bar{\mu}+\frac{2P_{2}\left( t\right)
\bar{\mu}+P_{3}\left( t\right) }{2P_{1}\left( t\right) }\right) , & \text{if
}x-\bar{\mu}+\frac{2P_{2}\left( t\right) \bar{\mu}+P_{3}\left( t\right) }{%
2P_{1}\left( t\right) }<0, \\
0, & \text{if }x-\bar{\mu}+\frac{2P_{2}\left( t\right) \bar{\mu}+P_{3}\left(
t\right) }{2P_{1}\left( t\right) }>0,%
\end{array}%
\right.  \label{optimalcontrol}
\end{eqnarray}%
Moreover, the value function can be shown%
\begin{equation}
\mathcal{V}\left( t,\mu \right) =\left\{
\begin{array}{ll}
\mathcal{V}_{1}\left( t,\mu \right) , & \text{if }x-\bar{\mu}+\frac{%
2P_{2}\left( t\right) \bar{\mu}+P_{3}\left( t\right) }{2P_{1}\left( t\right)
}<0; \\
\mathcal{V}_{2}\left( t,\mu \right) , & \text{if }x-\bar{\mu}+\frac{%
2P_{2}\left( t\right) \bar{\mu}+P_{3}\left( t\right) }{2P_{1}\left( t\right)
}>0.%
\end{array}%
\right.  \label{value}
\end{equation}
\end{theorem}

\begin{remark}
On $\Pi _{3},$ we conjecture $\mathcal{V}\left( t,\mu \right) =\min \left\{
\mathcal{V}_{1}\left( t,\mu \right) ,\mathcal{V}_{2}\left( t,\mu \right)
\right\} .$ In this case, the viscosity solution theory might be borrowed,
however this is beyond the scope of this article. We will consider this
issue in near future.
\end{remark}

\section{Application to Finance}

\label{sect4} In this section, we apply the general results established in
the previous section to a financial engineering. Suppose that a financial
market has $m+1$ assets evolved continuously on a finite horizon $[0,T]$. As
usual, one asset is a {bond (\emph{riskless})}, whose price denoted by $%
S_{0}(t),\;t\geq 0$, is driven by
\begin{equation}
\left\{
\begin{array}{rcl}
\mathrm{d}S_{0}\left( t\right) & = & rS_{0}\left( t\right) dt,\text{ }t\in %
\left[ 0,T\right] , \\
S_{0}\left( 0\right) & = & s_{0}>0,%
\end{array}%
\right.  \label{b}
\end{equation}%
where $r>0$ is the interest rate of the bond. The remaining $m$ assets are
stocks (risky), and their prices are described by
\begin{equation}
\left\{
\begin{array}{rcl}
\mathrm{d}S_{i}\left( t\right) & = & S_{i}\left( t\right) \left\{ b_{i}%
\mathrm{d}t+\sum_{j=1}^{m}\sigma _{ij}\mathrm{d}W^{j}\left( t\right)
\right\} ,\text{ }t\in \left[ 0,T\right] , \\
P_{i}\left( 0\right) & = & p_{i}>0,%
\end{array}%
\right.  \label{s}
\end{equation}%
where $b_{i}>r$ is the appreciation rate and $\sigma _{ij}$ is the
volatility coefficient. Denote $b:=(b_{1},\cdots ,b_{m})^{\prime }$ and $%
\sigma :=(\sigma _{ij})$. We assume throughout that $r$, $b$ and $\sigma $
are \emph{deterministic} constants. In addition, we impose that the
non-degeneracy condition
\begin{equation*}
\sigma \sigma ^{\prime }\geq \delta I,\quad
\end{equation*}%
where $\delta >0$ is a given constant, is satisfied. Also, we define the {%
relative risk coefficient}
\begin{equation*}
\theta \triangleq \sigma ^{-1}(b-r)\mathbf{1}),
\end{equation*}%
where $\mathbf{1}$ is the $m$-dimensional column vector with each component
equal to $1$.

Suppose an agent has an initial wealth $X_{0}>0$ and the total wealth of his
position at time $t\geq 0$ is $X(t)$, Then $X(t),$follows (see, e.g., \cite%
{ZhouLi})
\begin{equation}
\left\{
\begin{array}{rcl}
\mathrm{d}X\left( t\right) & = & \left\{ rX\left( t\right)
+\sum_{j=1}^{m}\left( b_{i}-r\right) u_{i}\right\} \mathrm{d}t \\
&  & +\sum_{j=1}^{m}\sum_{i=1}^{m}\sigma _{ij}u_{i}\left( t\right) \mathrm{d}%
W^{j}\left( t\right) ,\text{ }t\in \left[ 0,T\right] , \\
X\left( 0\right) & = & X_{0},%
\end{array}%
\right.  \label{w}
\end{equation}%
where $u_{i}(t),$ $i=0,1,\cdots ,m,$ denotes the total market value of the
agent's wealth in the $i$-th bond/stock. We call $u(t):=(u_{1}(t),\cdots
,u_{m}(t))$ the portfolio (which changes over time $t$). An important
restriction considered in this paper is the prohibition of short-selling the
stocks, i.e., it must be satisfied that $u_{i}(t)\geq 0$, $i=1,\cdots ,m$.
On the other hand, borrowing from the money market (at the interest rate $r$%
) is still allowed; that is, $u_{0}(t)$ is not explicitly constrained.

Mean-variance portfolio selection refers to the problem of finding an
allowable investment policy (i.e., a dynamic portfolio satisfying all the
constraints) such that the risk measured by
\begin{equation*}
J\left( u\right) =\alpha \mathbb{D}X(T)-\beta \mathbb{E}\left[ X(T)\right] +%
\mathbb{E}\left[ \int_{0}^{T}\left( \gamma X(t)^{2}-\gamma \left( \mathbb{E}%
\left[ X(t)\right] \right) ^{2}-\kappa X(t)\right) \mathrm{d}t\right] ,\text{
}
\end{equation*}%
where $\alpha >0,$ $\beta \geq 0,\gamma \geq 0,\kappa \geq 0$ and $\mathbb{D}%
X(T)$ denotes the variance of random variable $X(T),$ is minimized.

We recall the assumptions imposed in \cite{LZL2002}.

\begin{remark}
\label{ass:d} In \cite{LZL2002}, the authors assumed that the value of the
expected terminal wealth $d$ satisfies $d\geq X_{0}e^{rT}$, which means that
the investor's expected terminal wealth $d$ cannot be less than $X_{0}e^{rT}$
which coincides with the amount that he/she would earn if all of the initial
wealth is invested in the bond for the entire investment period. Otherwise,
the solution of the problem under $d<X_{0}e^{rT}$ seems to be foolish for
rational investors. In the current setting, the admissible controls belong
to a positive convex cone, so the value of the expected terminal wealth may
not be arbitrary. A natural question arises, of course, how to determine the
maximum value of $\mathbb{E}\left[ X(T)\right] $? This question also raised
in \cite{lx16}. Our next destination is to response this issue.
\end{remark}

\begin{definition}
\label{defi:control-positive} A portfolio $u(\cdot )$ is said to be
admissible if $u(\cdot )\in L_{\mathcal{F}}^{2}(0,T;\mathbb{R}_{+}^{m})$.
\end{definition}

\begin{definition}
The mean-variance portfolio selection problem is formulated as the following
optimization problem
\begin{equation}
\begin{array}{ll}
& \min J\left( u\right) , \\
\mbox{{\rm subject to}} & \left\{
\begin{array}{l}
u(\cdot )\in L_{\mathcal{F}}^{2}(0,T;\mathbb{R}_{+}^{m}), \\
(X(\cdot ),u(\cdot ))~\hbox{\rm admit
(\ref{w})}.%
\end{array}%
\right.%
\end{array}
\label{eq:utility-positive}
\end{equation}%
Moreover, the optimal control of (\ref{eq:utility-positive}) denoted by $%
u^{\ast }$ is called an efficient strategy, $J\left( u^{\ast }\right) $ is
the optimal value of (\ref{eq:utility-positive}) corresponding to $u^{\ast }$%
.
\end{definition}

\begin{remark}
We do not consider the the equality constraint $\mathbb{E}\left[ X(T)\right]
=d$ by introducing a Lagrange multiplier $\mu \in \mathbb{R}$ like \cite%
{LZL2002}.
\end{remark}

We now focus on the optimal control problem (\ref{eq:utility-positive}).

Set%
\begin{equation}
\left\{
\begin{array}{lll}
\mathrm{d}X\left( s\right) & = & \left[ AX_{s}+Bu\left( s\right) \right]
\mathrm{d}s+\sum_{j=1}^{m}D_{j}u\left( s\right) \mathrm{d}W^{j}\left(
s\right) , \\
X\left( t\right) & = & \xi ,\text{ with }\mu =\mathbb{P}_{\xi },%
\end{array}%
\right.  \label{sde3}
\end{equation}%
where $A=r,\bar{A}=0,B=\left( b_{1}-r,\ldots ,b_{m}-r\right)
,b_{0}=0,D_{j}=\left( \sigma _{1j},\ldots ,\sigma _{mj}\right) .$

Let
\begin{equation*}
\bar{\nu}=\arg \min_{\bar{\nu}\in \left[ 0,\infty \right) ^{m}}\frac{1}{2}%
\left\Vert \sigma ^{-1}\nu +\sigma ^{-1}\left( b-r\right) \mathbf{1}%
\right\Vert ^{2}
\end{equation*}%
and
\begin{equation*}
\bar{\theta}=\sigma ^{-1}\bar{\nu}+\sigma ^{-1}\left( b-r\right) \mathbf{1.}
\end{equation*}

We display
\begin{equation*}
\left\{
\begin{array}{l}
\dot{P}_{1}\left( t\right) +2rP_{1}\left( t\right) +\gamma -\left\Vert \bar{%
\theta}\right\Vert ^{2}=0, \\
P_{1}\left( T\right) =\alpha ,%
\end{array}%
\right.
\end{equation*}%
\begin{equation*}
\left\{
\begin{array}{l}
\dot{P}_{2}\left( t\right) -\frac{\left\Vert \bar{\theta}\right\Vert ^{2}}{%
P_{1}\left( t\right) }P_{2}^{2}\left( t\right) +2rP_{2}\left( t\right) =0,
\\
P_{2}\left( T\right) =0,%
\end{array}%
\right.
\end{equation*}%
\begin{equation*}
\left\{
\begin{array}{l}
\dot{P}_{3}\left( t\right) +\left( r-\frac{P_{2}\left( t\right) }{%
P_{1}\left( t\right) }\left\Vert \bar{\theta}\right\Vert ^{2}\right)
P_{3}\left( t\right) -\kappa =0, \\
P_{3}\left( T\right) =-\beta ,%
\end{array}%
\right.
\end{equation*}%
and
\begin{equation*}
\left\{
\begin{array}{l}
\dot{P}_{4}\left( t\right) -\frac{1}{4}\frac{P_{3}^{2}\left( t\right) }{%
P_{1}\left( t\right) }\left\Vert \bar{\theta}\right\Vert ^{2}=0, \\
P_{4}\left( T\right) =0,%
\end{array}%
\right.
\end{equation*}%
which can be explicitly solved such that
\begin{equation*}
\left\{
\begin{array}{l}
P_{1}\left( t\right) =\alpha e^{2r\left( T-t\right) }+\frac{\gamma
-\left\Vert \bar{\theta}\right\Vert ^{2}}{2r}\left( e^{2r\left( T-t\right)
}-1\right) >0, \\
P_{2}\left( t\right) =0, \\
P_{3}\left( t\right) =-\beta e^{r\left( T-t\right) }-\frac{\kappa }{r}\left(
e^{r\left( T-t\right) }-1\right) <0, \\
P_{4}\left( t\right) =-\frac{1}{4}\left\Vert \bar{\theta}\right\Vert
^{2}\int_{t}^{T}\frac{P_{3}^{2}\left( s\right) }{P_{1}\left( s\right) }%
\mathrm{d}s.%
\end{array}%
\right.
\end{equation*}%
We compute that, for $\forall \left( t,x,\mu \right) \in \left[ 0,T\right]
\times \mathbb{R}\times \mathcal{P}_{2}\left( \mathbb{R}\right) ,$
\begin{eqnarray}
&&x-\bar{\mu}+\frac{2P_{2}\left( t\right) \bar{\mu}+P_{3}\left( t\right) }{%
2P_{1}\left( t\right) }  \notag \\
&=&x-\bar{\mu}+\frac{P_{3}\left( t\right) }{2P_{1}\left( t\right) }  \notag
\\
&=&x-\bar{\mu}+\frac{-\beta e^{r\left( T-t\right) }-\frac{\kappa }{r}\left(
e^{r\left( T-t\right) }-1\right) }{2\left[ \alpha e^{2r\left( T-t\right) }+%
\frac{\gamma -\left\Vert \bar{\theta}\right\Vert ^{2}}{2r}\left( e^{2r\left(
T-t\right) }-1\right) \right] }.
\end{eqnarray}%
Let
\begin{equation*}
\Theta \left( t,x,\mu \right) =x-\bar{\mu}+\frac{-\beta e^{r\left(
T-t\right) }-\frac{\kappa }{r}\left( e^{r\left( T-t\right) }-1\right) }{2%
\left[ \alpha e^{2r\left( T-t\right) }+\frac{\gamma -\left\Vert \bar{\theta}%
\right\Vert ^{2}}{2r}\left( e^{2r\left( T-t\right) }-1\right) \right] }.
\end{equation*}%
Hence, if $\Theta \left( t,x,\mu \right) <0,$ it means
\begin{eqnarray}
\mathcal{V}\left( t,\mu \right) &=&\min_{u\left( \cdot \right) \in L_{%
\mathcal{F}}^{2}(0,T;\mathbb{R}_{+}^{m})}\Bigg \{\alpha \mathbb{D}X(T)-\beta
\mathbb{E}\left[ X(T)\right]  \notag \\
&&+\mathbb{E}\left[ \int_{0}^{T}\left( \gamma X(t)^{2}-\gamma \left( \mathbb{%
E}\left[ X(t)\right] \right) ^{2}-\kappa X(t)\right) \mathrm{d}t\right] %
\Bigg \}  \notag \\
&=&\mathcal{V}_{1}\left( t,\mu \right)  \notag \\
&=&\text{Var}\left( \mu \right) \left( P_{1}\left( t\right) \right)
+P_{2}\left( t\right) \bar{\mu}^{2}+P_{3}\left( t\right) \bar{\mu}%
+P_{4}\left( t\right)  \notag \\
&=&\left[ \alpha e^{2r\left( T-t\right) }+\frac{\gamma +\left\Vert \bar{%
\theta}\right\Vert ^{2}}{2r}\left( e^{2r\left( T-t\right) }-1\right) \right]
\mathbb{D}\xi  \notag \\
&&-\left[ \beta e^{r\left( T-t\right) }+\frac{\kappa }{r}\left( e^{r\left(
T-t\right) }-1\right) \right] \mathbb{E}\xi  \notag \\
&&-\frac{1}{4}\left\Vert \bar{\theta}\right\Vert ^{2}\int_{0}^{T}\frac{\left[
\beta e^{r\left( T-s\right) }+\frac{\kappa }{r}\left( e^{r\left( T-s\right)
}-1\right) \right] ^{2}}{\alpha e^{2r\left( T-s\right) }+\frac{\gamma
+\left\Vert \bar{\theta}\right\Vert ^{2}}{2r}\left( e^{2r\left( T-s\right)
}-1\right) }\mathrm{d}s.  \label{g1}
\end{eqnarray}%
The associated average optimal strategy can be expressed as, for $\forall
s\in \left[ t,T\right] ,$ if $X^{\ast }\left( s\right) -\mathbb{E}\left[
X^{\ast }\left( s\right) \right] +\frac{2P_{2}\left( s\right) \bar{\mu}%
+P_{3}\left( s\right) }{2P_{1}\left( s\right) }<0$
\begin{eqnarray*}
u^{\ast }\left( s,X^{\ast }\left( s\right) ,\mathbb{P}_{X^{\ast }\left(
s\right) }\right) &=&-\sigma ^{-1}\bar{\theta}\cdot \left( X^{\ast }\left(
s\right) -\mathbb{E}\left[ X^{\ast }\left( s\right) \right] +\frac{%
2P_{2}\left( t\right) \bar{\mu}+P_{3}\left( t\right) }{2P_{1}\left( t\right)
}\right) \\
&=&-\sigma ^{-1}\bar{\theta}\cdot \left( X^{\ast }\left( s\right) -\mathbb{E}%
\left[ X^{\ast }\left( s\right) \right] +\frac{P_{3}\left( t\right) }{%
2P_{1}\left( t\right) }\right) \\
&=&\sigma ^{-1}\bar{\theta}\cdot \left( X^{\ast }\left( s\right) -\mathbb{E}%
\left[ X^{\ast }\left( s\right) \right] +\frac{\beta e^{r\left( T-t\right) }-%
\frac{\kappa }{r}\left( e^{r\left( T-t\right) }-1\right) }{2\left[ \alpha
e^{2r\left( T-t\right) }+\frac{\gamma }{2r}\left( e^{2r\left( T-t\right)
}-1\right) \right] }\right) .
\end{eqnarray*}%
Particularly, at time $t=0$, if a investor possesses a deterministic wealth $%
\xi =X_{0}$, of course, its variance $\mathbb{D}\xi =0.$ Then, it follows
that
\begin{eqnarray*}
&&\min_{u\left( \cdot \right) \in L_{\mathcal{F}}^{2}(0,T;\mathbb{R}%
_{+}^{m})}\Bigg \{\alpha \mathbb{D}X(T)-\beta \mathbb{E}\left[ X(T)\right] \\
&&+\mathbb{E}\left[ \int_{0}^{T}\left( \gamma X(t)^{2}-\gamma \left( \mathbb{%
E}\left[ X(t)\right] \right) ^{2}-\kappa X(t)\right) \mathrm{d}t\right] %
\Bigg \} \\
&=&-\left[ \beta e^{r\left( T-t\right) }+\frac{\kappa }{r}\left( e^{r\left(
T-t\right) }-1\right) \right] X_{0} \\
&&-\frac{1}{4}\left\Vert \bar{\theta}\right\Vert ^{2}\int_{0}^{T}\frac{\left[
\beta e^{r\left( T-s\right) }+\frac{\kappa }{r}\left( e^{r\left( T-s\right)
}-1\right) \right] ^{2}}{\alpha e^{2r\left( T-s\right) }+\frac{\gamma }{2r}%
\left( e^{2r\left( T-s\right) }-1\right) }\mathrm{d}s.
\end{eqnarray*}%
If we consider $\alpha =1,\beta =0,$ (\ref{eq:utility-positive}) follows
that
\begin{equation}
\min_{u\left( \cdot \right) \in L_{\mathcal{F}}^{2}(0,T;\mathbb{R}_{+}^{m})}%
\left[ \mathbb{D}X(T)\right] =0.  \label{par1}
\end{equation}

We are now arriving at discussion of the boundedness of $\mathbb{E}\left[
X^{\ast }\left( T\right) \right] $, \emph{namely}, the capital market line.
Observe that the optimal control (\ref{optimalcontrol}) is composed of two
parts. So we will investigate these two items one by one.

\noindent \textbf{Case 1.} The controlled process $\mathbb{E}\left[
X_{s}^{\ast }\right] $ with
\begin{equation*}
\mathbb{E}\left[ u^{\ast }\left( s,X^{\ast }\left( s\right) ,\mathbb{P}%
_{X^{\ast }\left( s\right) }\right) \right] =\sigma ^{-1}\bar{\theta}\cdot
\frac{\beta e^{r\left( T-t\right) }-\frac{\kappa }{r}\left( e^{r\left(
T-t\right) }-1\right) }{2\left[ \alpha e^{2r\left( T-t\right) }+\frac{\gamma
}{2r}\left( e^{2r\left( T-t\right) }-1\right) \right] }.
\end{equation*}%
reads%
\begin{equation}
\left\{
\begin{array}{lll}
\mathrm{d}\mathbb{E}X^{\ast }\left( s\right)  & = & \left[ r\mathbb{E}\left[
X^{\ast }\left( s\right) \right] +B\sigma ^{-1}\bar{\theta}\cdot \frac{\beta
e^{r\left( T-t\right) }-\frac{\kappa }{r}\left( e^{r\left( T-t\right)
}-1\right) }{2\left[ \alpha e^{2r\left( T-t\right) }+\frac{\gamma }{2r}%
\left( e^{2r\left( T-t\right) }-1\right) \right] }\right] \mathrm{d}s, \\
X^{\ast }\left( 0\right)  & = & X_{0}.%
\end{array}%
\right.   \label{sde4}
\end{equation}%
Set
\begin{equation*}
p\left( t\right) =B\sigma ^{-1}\bar{\theta}\cdot \frac{\beta e^{r\left(
T-t\right) }-\frac{\kappa }{r}\left( e^{r\left( T-t\right) }-1\right) }{2%
\left[ \alpha e^{2r\left( T-t\right) }+\frac{\gamma }{2r}\left( e^{2r\left(
T-t\right) }-1\right) \right] }.
\end{equation*}%
After simple calculation, it yields, for any $s\in \left[ t,T\right] ,$%
\begin{equation}
\mathbb{E}\left[ X^{\ast }\left( s\right) \right] =e^{rs}\left(
X_{0}+\int_{0}^{s}p\left( z\right) e^{-rz}\mathrm{d}z\right) .  \label{m1}
\end{equation}%
\textbf{Case 2.} Similarly, if $u^{\ast }\left( s,X^{\ast }\left( s\right) ,%
\mathbb{P}_{X^{\ast }\left( s\right) }\right) =0,$ we have%
\begin{equation}
\mathbb{E}\left[ X^{\ast }\left( s\right) \right] =e^{rs}X_{0}.  \label{m2}
\end{equation}

From (\ref{m1}) and (\ref{m2}), we assert that mean of future return $%
\mathbb{E}\left[ X(T)\right] $ satisfies that
\begin{equation}
e^{rT}X_{0}\leq \mathbb{E}\left[ X(T)\right] \leq e^{rT}\left(
X_{0}+\int_{0}^{T}p\left( z\right) e^{-rz}\mathrm{d}z\right) .  \label{m3}
\end{equation}%
under the short-selling of stocks prohibited.

The above discussion leads to the following theorem.

\begin{theorem}
At time $t\geq 0,$ if a investor possesses a random wealth $\xi \in L^{2}(%
\mathcal{G\vee F}_{t};\mathbb{R}^{n})$. Then, the average optimal strategy
of portfolio selection problem (\ref{eq:utility-positive}) can be written
as, for $\forall s\in \left[ t,T\right] $%
\begin{equation}
u^{\ast }\left( s,X^{\ast }\left( s\right) ,\mathbb{P}_{X^{\ast }\left(
s\right) }\right) =\sigma ^{-1}\bar{\theta}\cdot \left( X^{\ast }\left(
s\right) -\mathbb{E}\left[ X^{\ast }\left( s\right) \right] +\frac{\beta
e^{r\left( T-t\right) }-\frac{\kappa }{r}\left( e^{r\left( T-t\right)
}-1\right) }{2\left[ \alpha e^{2r\left( T-t\right) }+\frac{\gamma }{2r}%
\left( e^{2r\left( T-t\right) }-1\right) \right] }\right) .  \label{t1}
\end{equation}%
Moreover,
\begin{eqnarray}
&&\min_{u\left( \cdot \right) \in L_{\mathcal{F}}^{2}(0,T;\mathbb{R}%
_{+}^{m})}\Bigg \{\alpha \mathbb{D}X(T)-\beta \mathbb{E}\left[ X(T)\right]
\notag \\
&&+\mathbb{E}\left[ \int_{t}^{T}\left( \gamma X(t)^{2}-\gamma \left( \mathbb{%
E}\left[ X(t)\right] \right) ^{2}-\kappa X(t)\right) \mathrm{d}t\right] %
\Bigg \}  \notag \\
&=&\left[ \alpha e^{2r\left( T-t\right) }+\frac{\gamma }{2r}\left(
e^{2r\left( T-t\right) }-1\right) \right] \mathbb{D}\xi   \notag \\
&&-\left[ \beta e^{r\left( T-t\right) }+\frac{\kappa }{r}\left( e^{r\left(
T-t\right) }-1\right) \right] \mathbb{E}\xi   \notag \\
&&-\frac{1}{4}\left\Vert \bar{\theta}\right\Vert ^{2}\int_{t}^{T}\frac{\left[
\beta e^{r\left( T-s\right) }+\frac{\kappa }{r}\left( e^{r\left( T-s\right)
}-1\right) \right] ^{2}}{\alpha e^{2r\left( T-s\right) }+\frac{\gamma }{2r}%
\left( e^{2r\left( T-s\right) }-1\right) }\mathrm{d}s.  \label{t2}
\end{eqnarray}%
The capital market line $\mathbb{E}\left[ X(T)\right] $ satisfies
\begin{equation}
e^{rT}X_{0}\leq \mathbb{E}\left[ X(T)\right] \leq e^{rT}\left(
X_{0}+\int_{0}^{T}p\left( z\right) e^{-rz}\mathrm{d}z\right) .  \label{t3}
\end{equation}
\end{theorem}

\section{Concluding remark}

\label{sect5} To conclude this paper, let us make some remarks. In this
paper, we have presented some results on the indefinite stochastic
McKean-Vlasov LQ problem with deterministic coefficients. The optimal
control can be represented as a state feedback form via the solutions of two
Riccati equations and the distribution of $\xi .$ We apply our theoretic
results to study the mean-variance problem under a short-selling prohibition
and to obtain the investment risk and the capital market line. There are
still some interesting extensions deserved attention, for instance, the
coefficients can be random, which is close to reality; Besides, note that
the at time $t$ and when total wealth distribution is $\mu $ , the optimal
dollar amount $u(t,\mu )$ invested in the risky asset is of the form (\ref%
{t1}). In particular, this implies that the dollar amount invested in the
risky asset does not depend on current wealth $\xi $ via its distribution.
This phenomenon is unreasonable from view point of economics, since it
implies that you will invest the same number of dollars in the stock if your
wealth is $100$ dollars as you would if your wealth is ten million dollars.
The reason for this anomaly is the fact that the risk aversion parameter is
assumed to be $1$, which is impractical. A person's risk preference
certainly depends on how wealthy he owns; and hence the obvious implication
is that we should explicitly allow a function $\gamma $ to depend on current
wealth's distribution, that is $\gamma \left( \mu \right) $. We will study
the mean-variance problems with a state dependent risk aversion in our
future work. Besides, in the literature, we name $W^{0}$ as a common noise
affecting the whole setting, whilst $W$ is the idiosyncratic chaos for the
random field and all processes within, (for further details see discussion
in (\cite{cd17}, Section 4.3])). Pham and Wei \cite{PW2017} have derived the
HJB equation for a general model with common noise. We will also consider
this framework in near future.

\section{Appendix}

\label{sect6} 
\label{app}

\subsection{Technique Lemmas}

We first recall some results from convex analysis from \cite{Rock}.

\begin{lemma}
\label{lem:s1} Let $s$ be a continuous, strictly convex quadratic function
\begin{equation}
\begin{array}{l}
s(z)\triangleq \frac{1}{2}\left\Vert ( \mathcal{D}^{\prime }) ^{-1}z+(
\mathcal{D}^{\prime }) ^{-1}\mathcal{B}^{\prime }\right\Vert ^{2}%
\end{array}
\label{eq:s1}
\end{equation}%
over $z\in \lbrack 0,\infty )^{m}$, where $\mathcal{B}^{\prime }\in \mathbb{R%
}_{+}^{m}$, $\mathcal{D}\in \mathbb{R}^{m\times m}$ and $\mathcal{D}^{\prime
}\mathcal{D}>0$. Then $s$ has a unique minimizer $\bar{z}\in \lbrack
0,\infty )^{m}$, i.e.,
\begin{equation*}
\left\Vert ( \mathcal{D}^{\prime }) ^{-1}\bar{z}+( \mathcal{D}^{\prime })
^{-1}\mathcal{B}^{\prime }\right\Vert ^{2}\leq \left\Vert ( \mathcal{D}%
^{\prime }) ^{-1}z+( \mathcal{D}^{\prime }) ^{-1}\mathcal{B}^{\prime
}\right\Vert ^{2},\quad \forall z\in \lbrack 0,\infty )^{m}.
\end{equation*}
\end{lemma}

The Kuhn-Tucker conditions for the minimization of $s$ in (\ref{eq:s1}) over
$[0, \infty)^m$ lead to the {Lagrange multiplier} vector $\bar\nu \in [0,
\infty)^m$ such that $\bar\nu = \nabla s(\bar z) = (\mathcal{D}^{\prime }%
\mathcal{D})^{-1}\bar z + (\mathcal{D}^{\prime }\mathcal{D})^{-1}\mathcal{B}%
^{\prime }$ and $\bar\nu^{\prime }\bar z = 0$.

\vskip 24pt

\begin{lemma}
\label{lem:h} Let $h$ be a continuous, strictly convex quadratic function
\begin{equation*}
\begin{array}{l}
h(z)\triangleq \frac{1}{2}z^{\prime }\mathcal{D}^{\prime }\mathcal{D}%
z-\alpha \mathcal{B}z%
\end{array}%
\end{equation*}
over $z\in \lbrack 0,\infty )^{m}$, where $\mathcal{B}^{\prime }\in \mathbb{R%
}_{+}^{m}$, $\mathcal{D}\in \mathbb{R}^{m\times m}$ and $\mathcal{D}^{\prime
}\mathcal{D}>0$.

\begin{itemize}
\item[\textrm{(i)}] For every $\alpha >0$, $h$ has the unique minimizer $%
\alpha\mathcal{D}^{-1}\bar{\xi}\in \lbrack 0,\infty )^{m}$, where $\bar{\xi}%
=(\mathcal{D}^{\prime -1}\bar{z}+(\mathcal{D}^{\prime -1}\mathcal{B}^{\prime
}$. Here $\bar{z}$ is the minimizer of $s(z)$ specified in Lemma \ref{lem:s1}%
. Furthermore, $\bar{z}^{\prime }\mathcal{D}^{-1}\bar{\xi}=0$ and
\begin{equation*}
\begin{array}{l}
h(\alpha \bar{\nu})=h(\alpha \mathcal{D}^{-1}\bar{\xi})=-\frac{1}{2}\alpha
^{2}\Vert \bar{\xi}\Vert ^{2}.%
\end{array}%
\end{equation*}
\vspace{-1cm}

\item[\textrm{(ii)}] For every $\alpha < 0$, $h$ has the unique minimizer $0$%
.
\end{itemize}
\end{lemma}

Lemma \ref{lem:s1} and Lemma \ref{lem:h}-(i) are proved in Section 5.2 and
Lemma 3.2 of \cite{XuShreve2}, while Lemma \ref{lem:h}-(ii) is obvious.

\begin{remark}
\label{rem:h} Note that the vector $\bar{\xi}$ is independent of the
parameter $\alpha $.
\end{remark}

\begin{lemma}
\label{gron}Given a real valued function $g\geq 0$ and a integrable
real-valued functions $h$, if there exists a constant $K>0,$ for any $t\in %
\left[ 0,T\right] ,$ such that
\begin{equation*}
g\left( t\right) \leq h\left( t\right) +K\int_{t}^{T}g\left( s\right)
\mathrm{d}s.
\end{equation*}%
Then
\begin{equation*}
g\left( t\right) \leq h\left( t\right) +K\int_{t}^{T}e^{K\left( s-t\right)
}h\left( s\right) \mathrm{d}s.
\end{equation*}
\end{lemma}

\noindent \textbf{Acknowledgement.} The authors thank Prof. Xunyu Zhou for
his suggestion and comments. The partial work was completed when the second
author was visiting the Department of Applied Mathematics, The Hong Kong
Polytechnic University. Their hospitality is very appreciate.

\end{document}